\documentclass[11pt,a4paper]{amsart}

\usepackage{amsmath}
\usepackage{amsfonts}
\usepackage{amssymb}

\usepackage{amsxtra}
\usepackage{amsthm}
\usepackage[dvipsnames]{xcolor}
\theoremstyle{plain}
\newcommand{\C}{\mathbb{C}}
\newcommand{\F}{\mathcal{F}}

\newcommand{\sing}{\textsf{Sing}}

\newcommand{\ord}{\textsf{ord}}

\newcommand{\tang}{{\rm tang}}

\newcommand{\mc}[1]{\mathcal{#1}}

\newtheorem{theorem}{Theorem}[section]
\newtheorem{maintheorem}{Theorem}

\newtheorem{secondtheorem}{Theorem}[section]

\newtheorem{lemma}[theorem]{Lemma}
\newtheorem{proposition}[theorem]{Proposition}
\newtheorem{corollary}[theorem]{Corollary}
\newtheorem{maincorollary}{Corollary}

\newtheorem{secondcorollary}{Corollary}

\theoremstyle{definition}
\newtheorem{definition}{Definition}[section]
\newtheorem{example}{Example}[section]
\newtheorem{remark}{Remark}[section]

\DeclareSymbolFont{matha}{OML}{txmi}{m}{it}
\DeclareMathSymbol{\varv}{\mathord}{matha}{118}
\usepackage[pagebackref]{hyperref}
\begin{document}

\title[The Bruce--Roberts number of 1-forms along complex varieties]{The Bruce--Roberts number of holomorphic 1-forms along complex analytic varieties}

\author{Pedro Barbosa}
\address[Pedro Barbosa]{Departamento de Matem\'atica - ICEX, Universidade Federal de Minas Gerais, UFMG}
\curraddr{Av. Pres. Ant\^onio Carlos 6627, 31270-901, Belo Horizonte-MG, Brasil.}
\email{pedrocbj@ufmg.br}
\author{Arturo Fern\'andez-P\'erez}
\address[A. Fern\'andez-P\'erez]{Departamento de Matem\'atica - ICEX, Universidade Federal de Minas Gerais, UFMG}
\curraddr{Av. Pres. Ant\^onio Carlos 6627, 31270-901, Belo Horizonte-MG, Brasil.}
\email{fernandez@ufmg.br}

\author{V\'ictor Le\'on}
\address[V. Le\'on]{ILACVN - CICN, Universidade Federal da Integra\c c\~ao Latino-Americana, UNILA}
\curraddr{Parque tecnol\'ogico de Itaipu, Foz do Igua\c cu-PR, 85867-970 - Brasil}
\email{victor.leon@unila.edu.br}

\thanks{The first author is supported by CAPES-Brazil. The second author acknowledges support from CNPq Projeto Universal 408687/2023-1 ``Geometria das Equa\c{c}\~oes Diferenciais Alg\'ebricas" and CNPq-Brazil PQ-306011/2023-9.}

\begin{abstract}
We introduce the notion of the \textit{Bruce--Roberts number} for holomorphic 1-forms relative to complex analytic varieties.  Our main result shows that the Bruce--Roberts number of a 1-form $\omega$ relative to a complex analytic hypersurface $X$ with an isolated singularity can be expressed in terms of the \textit{Gom\'ez-Mont--Seade--Verjovsky index} of $\omega$ along $X$, the \textit{Milnor number} of $\omega$ and the \textit{Tjurina number} of $X$. This result allows us to recover known formulas for the Bruce--Roberts number of a holomorphic function along $X$ and to establish connections between this number, the radial index, and the local Euler obstruction of $\omega$ along $X$. Moreover, we present applications to both global and local holomorphic foliations in complex dimension two.
\end{abstract}

\maketitle
\tableofcontents
\section{Introduction and statement of the results}

Let $(X, 0)$ denote the germ of a complex analytic variety in $(\mathbb{C}^n,0)$, and let $f: (\mathbb{C}^n, 0) \to (\mathbb{C}, 0)$ be the germ of a holomorphic function in $(\mathbb{C}^n,0)$. The \textit{Bruce--Roberts number} associated with $f$ relative to $(X, 0)$ was originally introduced by J.~W. Bruce and R. M. Roberts in their seminal work \cite{BR}. This number, denoted by $\mu_{BR}(f,X)$, is defined as follows:
\[
\mu_{BR}(f,X) = \dim_{\mathbb{C}} \frac{{\mathcal{O}_n}}{{d f(\Theta_X)}},
\]
where $\mathcal{O}_n$ represents the local ring of holomorphic functions from $(\mathbb{C}^n, 0)$ to $(\mathbb{C},0)$, $df$ stands for the differential of $f$, and $\Theta_X$ is the $\mathcal{O}_n$-submodule of $\Theta_n$ consisting of holomorphic vector fields on $(\mathbb{C}^n, 0)$ that are tangent to $(X, 0)$ over their regular points \cite{Saito1}. If $I_X\subset\mathcal{O}_n$ is the ideal of germs of holomorphic functions vanishing on $(X,0)$, then
\[
\Theta_X =\{\xi \in \Theta_n:\; dh(\xi)\in I_X,\;\forall\,\,h \in I_X\}.
\]
In particular, when $X = \mathbb{C}^n$, $df(\Theta_n)$ corresponds to the Jacobian ideal $Jf$ of $f$ which is generated by the partial derivatives of $f$ in $\mathcal{O}_n$.
Consequently, $\mu_{BR}(f, \C^n)$ coincides with the classical \textit{Milnor number} $\mu_0(f)$ of $f$. We recall that $\mu_0(f)$ is defined algebraically as the colength of the Jacobian ideal $Jf$ in $\mathcal{O}_n$.
Furthermore, if $X$ is the germ of a complex analytic subvariety in $(\mathbb{C}^n,0)$, then  $\mu_{BR}(f,X)$ is finite if and only if $f$ has an isolated singularity over $(X,0)$. 
The Bruce--Roberts number for holomorphic functions has been studied by many authors, see for instance \cite{Ahmed},  \cite{Nivaldo2009}, \cite{Nuno2013}, \cite{Nuno2020},  \cite{Bivia2020},  \cite{Tajima2021}, \cite{Lima2021}, \cite{Kourliouros2021}, \cite{Lima2024},  and \cite{Bivia2024}.
\par The purpose of this paper is to extend the definition of the \textit{Bruce--Roberts number} to \textit{holomorphic 1-forms relative to complex analytic varieties}. 
Our objective is to establish connections between the Bruce--Roberts number and 
 other well-known indices of 1-forms, such as the G\'omez-Mont–Seade–Verjovsky index \cite{Ebeling}, the radial index, the local Euler obstruction, and the Milnor number (Poincar\'e-Hopf type index). As a result, we offer a reinterpretation of these indices from the perspective of \textit{Foliation theory}. Furthermore, we present applications to both \textit{Foliations} and \textit{Singularity theory}, with a particular focus on complex dimension two. More precisely, let $\omega$ be the germ of a holomorphic 1-form with an isolated singularity at $0\in\mathbb{C}^n$, $n\geq 2$, and let $X$ be a germ of complex analytic variety with an isolated singularity at $0\in\mathbb{C}^n$. We define the \textit{Bruce--Roberts number} of the 1-form $\omega$ relative to $X$ as
\begin{equation*}\label{def_1}
\mu_{BR}(\omega, X):= \dim_{\mathbb{C}}\dfrac{\mathcal{O}_n}{\omega(\Theta_X)}.
\end{equation*}
Note that $\mu_{BR}(\omega,X)$ is finite if and only if the restriction of $\omega$ to $X_{\mathrm{reg}}$, the regular part of $X$, has no zeros in a punctured neighborhood of the origin; that is, $\omega_p|_{T_pX}\neq 0$ for every $p\in X_{\mathrm{reg}}\setminus\{0\}$ sufficiently close to $0$. We recall that $X$ is said to be \textit{invariant} by $\omega$ if $T_pX\subset\ker(\omega_p)$ at every regular point $p$ of $X$. When $X\subset(\C^2,0)$ is a reduced curve germ, the above finiteness condition is equivalent to requiring that none of its irreducible components be invariant by $\omega$. For higher-dimensional germs, non-invariance alone does not guarantee finiteness. Moreover, if $\omega=df$ for some $f\in\mathcal{O}_n$, then $\mu_{BR}(\omega,X)=\mu_{BR}(f,X)$.
\par For a germ of holomorphic 1-form $\omega=\displaystyle\sum^n_{i=1}A_i(x)dx_i$ with an isolated singularity at $0\in\C^n$, $n\geq 2$, we define the \textit{Milnor number} of $\omega$ at $0\in\C^n$ as 
\begin{equation*}\label{Milnor}
\mu_0(\omega):=\dim_\C \frac{\mathcal{O}_n}{\langle A_1,\ldots, A_n\rangle}.
\end{equation*}
\par It is worth noting that in the literature, the Milnor number of $\omega$ at $0\in\C^n$ is sometimes denoted as $\operatorname{ind}(\omega;\C^n,0)$, and referred to as the \textit{Poincar\'e-Hopf index} of $\omega$ at $0\in\C^n$; see, for instance, \cite[Definition 5.2.3]{Ebeling2023}. 
\par Analogous to the $GSV$-index (G\'omez-Mont--Seade--Verjovsky \cite{GSV}) for vector fields, Ebeling--Gusein-Zade introduced the notion of the $GSV$-index for a 1-form $\omega$ relative to $X$ in \cite{Ebeling2001}, \cite{Ebeling}. This index is denoted by $\operatorname{Ind}_{\operatorname{GSV}}(\omega;X,0)$ and is related to the \textit{radial index} $\operatorname{Ind}_{\operatorname{rad}}(\omega;X,0)$ defined in \cite{Ebeling2004}, and the \textit{local Euler obstruction} $\operatorname{Eu}(\omega;X,0)$ defined in \cite{Ebeling2005}.
For further details, precise definitions, and additional results concerning these indices, see, for instance, Sections \ref{index1} and \ref{Euler}.
\par With this terminology, we can now state our first result:
\begin{maintheorem}\label{mainresult} 
Let $\omega$ be a germ of a holomorphic 1-form with an isolated singularity at $0\in\mathbb{C}^n$, $n\geq 2$, and let $X$ be a reduced hypersurface germ with an isolated singularity at $0\in\C^n$.
 Assume that $\mu_{BR}(\omega,X)<\infty$. Then
\[ \mu_{BR}(\omega,X)=\operatorname{Ind}_{\operatorname{GSV}}(\omega;X,0)+\mu_0(\omega)-\tau_0(X),\]
where $\tau_0(X)$ is the \textit{Tjurina number} of $(X,0)$.  
\end{maintheorem}
\par We recall that if $X$ is defined by $\phi:(\C^n,0)\to(\C,0)$, the Tjurina number of $X$ is given by
\[\tau_0(X)=\dim_{\C}\frac{\mathcal{O}_n}{\langle \phi,\frac{\partial{\phi}}{\partial{x}_1},\ldots,\frac{\partial{\phi}}{\partial{x}_n}\rangle}.\]
\par It is important to note that in some cases, determining $\Theta_X$ explicitly can be difficult. However, Theorem \ref{mainresult} provides a straightforward method for computing $\mu_{BR}(\omega,X)$. \par For a fixed hypersurface germ $X$, Theorem \ref{mainresult} shows that $\mu_{BR}(\omega,X)$ is constant in any family for
which both $\mu_0(\omega)$ and $\operatorname{Ind}_{\operatorname{GSV}}(\omega;X,0)$ remain constant. In the exact case, the corresponding
invariance under $C^{0}
-\mathcal{R}_X$-equivalence follows from \cite[Corollary 4.3]{Nuno2020}.
\par As a corollary, we recover one of the main results from Nu\~{n}o-Ballesteros--Or\'efice-Okamoto--Pereira--Tomazella \cite[Corollary 4.1]{Nuno2020}:
\begin{maincorollary}\label{primercoro}
Let $X$ be an isolated hypersurface singularity defined by
$\phi :(\mathbb{C}^n,0)\to (\mathbb{C}, 0)$ and let $f\in\mathcal{O}_n$ satisfy $\mu_{BR}(f,X)<\infty$. Then
\[\mu_{BR}(f,X)=\mu_0( f ) + \mu_0(X)+\mu_0(\phi, f ) -\tau_0(X),\]
where $\mu_0(X)$ is the Milnor number of $X$ and $\mu_0(\phi,f)$ is the Milnor number of the isolated complete intersection singularity defined by $(\phi,f)$, in the sense of Hamm \cite{Hamm}.  
\end{maincorollary}
\par By combining Theorem \ref{mainresult}, Propositions \ref{radial_gsv} and \ref{euler_radial}, we obtain the following corollary:
\begin{secondcorollary}\label{segundocoro}
    Let $\omega$ be a germ of a holomorphic 1-form with isolated singularity at $0\in\mathbb{C}^n$, $n\geq 2$. Let $X$ be a germ of a complex analytic hypersurface with an isolated singularity at $0\in\C^n$.
 Assume that $\mu_{BR}(\omega,X)<\infty$. Then
 \[\mu_{BR}(\omega,X)-\mu_0(\omega)+\tau_0(X)-\mu_0(X)-\operatorname{Eu}(\omega;X,0)=(-1)^{n-2}\bar{\chi}(M_{\ell}),\]
 where $M_{\ell}$ is the Milnor fiber of the generic linear function $\ell:\mathbb{C}^n\to\mathbb{C}$ on $X$ and $\bar{\chi}(M_{\ell})=\chi(M_{\ell})-1$. 
\end{secondcorollary}
\par Following the definitions and topics explored in \cite{Lima2021}, we define the \textit{relative Bruce--Roberts number} of the 1-form $\omega$ relative to $X$ as
\begin{equation*}
\mu_{BR}^{-}(\omega,X):=\dim_{\C}\frac{\mathcal{O}_n}{\omega(\Theta_X)+I_X}, 
\end{equation*}
where $I_X\subset\mathcal{O}_n$ is the ideal of germs of holomorphic functions vanishing on $(X,0)$. 
\par A natural question arises regarding the relationship between $\mu_{BR}(\omega,X)$ and $\mu_{BR}^-(\omega,X)$.
The following theorem addresses this question:
\begin{secondtheorem}\label{correlthmBr}
Let $\omega$ be a germ of a holomorphic 1-form with an isolated singularity at $0\in\mathbb{C}^n$, $n\geq 2$. Let $X$ be a germ of a complex analytic hypersurface with an isolated singularity at $0\in\C^n$.
 Assume that $\mu_{BR}(\omega,X)<\infty$. Then
    \[\mu_{BR}(\omega,X)=\mu_0(\omega)+\mu_{BR}^-(\omega,X).\]
\end{secondtheorem}
Note that if we set $\omega=df$ in the preceding theorem, we recover the formula
\[\mu_{BR}(f,X)=\mu_0(f)+\mu^-_{BR}(f,X)\]
as found in \cite[Section 3]{Lima2021}.
\par Section \ref{foliations} is dedicated to applying our results to germs of holomorphic foliations on $(\C^2,0)$.
The first observation is that, for a germ of holomorphic foliation $\F$ defined by $\omega\in\Omega^{1}(\C^2,0)$, and for $(X,0)$ a germ of a complex analytic curve such that none of the irreducible components of $X$ is invariant by $\F$, we have 
\begin{equation}\label{eq_2}
\operatorname{Ind}_{\operatorname{GSV}}(\omega;X,0)=\tang(\F,X,0),
\end{equation}
where $\tang(\F,X,0)$ denotes the \textit{tangency order} of $\F$ along $X$ defined by Brunella \cite[Chapter 2, Section 2]{Brunella-book}. Then
\[\mu_{BR}(\F,X)=\mu_0(\F)+\operatorname{tang}(\F,X,0)-\tau_0(X),\]
by Theorem \ref{mainresult}, see for instance Corollary \ref{dim2}.  Continuing in Proposition \ref{tangirred}, we obtain 
\[  \operatorname{tang}(\F,X,0)=\text{ord}_{t=0}\  \varphi^* \omega+\mu_0(X),\]
where $\varphi$ is a Puiseux parametrization of $X$.  Combining  this with the equation (\ref{eq_2}) and Proposition \ref{radial_gsv}, we get 
\[ \operatorname{Ind}_{\operatorname{rad}}(\omega;X,0)=\text{ord}_{t=0}\  \varphi^* \omega.\]
This provides a straightforward interpretation of the radial index in dimension two. We also present blow-ups formulas for $\mu_{BR}(\omega,X)$ and $\mu_{BR}^-(\omega,X)$ and provide several examples. Additionally, we establish a characterization of when a non-dicritical foliation $\F$ is a generalized curve foliation in terms of $\mu_{BR}(\F,X)$ and $\mu_{BR}^-(\F,X)$; see Corollary \ref{generalized}. Finally, in Section \ref{global}, we derive a global formula for the sum of the Bruce--Roberts numbers of a global foliation on a compact complex surface $S$, see Theorem \ref{global_1}. As a consequence, we obtain an upper bound to the global Tjurina number of $X$ in terms of the characteristic numbers of $S$, $\F$ and $X$.
\par After the first version of this work was posted on arXiv (arXiv:2409.01237), Lima-Pereira, Nuño-Ballesteros, Oréfice-Okamoto, and Tomazella extended the study of Bruce-–Roberts numbers of holomorphic \(1\)-forms to isolated complete intersection singularities (ICIS) \cite{Lima2025}. The present paper focuses on the hypersurface case and develops applications to local and global holomorphic foliations on complex surfaces. For related developments concerning logarithmic derivations and relative invariants of analytic maps, see~\cite{Bivia2025,Ruas2025}.

\section{Gom\'ez-Mont--Seade--Verjovsky index of a holomorphic 1-form}\label{index1}
In this section, we present definitions and results related to the index of a holomorphic 1-form relative to an isolated complete intersection singularity. As noted in \cite[p. 133]{Ebeling2006}, the concept of considering indices of holomorphic 1-forms was first proposed by V.I. Arnold \cite{Arnold}. 
In this paper, we use the algebraic definition of this index specifically for the case where $X$ is a complex hypersurface with an isolated singularity at $0\in\C^n$. For a more comprehensive account of this theory, the reader can refer to \cite{Ebeling}.  
 
Let $\omega$ be a germ of a holomorphic $1$-form, given by $\omega=\displaystyle\sum_{i=1}^{n+k} A_i(x)dx_i$, where $A_i \in \mathcal{O}_{n+k}$ in $(\C^{n+k},0)$. Let $(V,0) \subset (\C^{n+k},0)$ be an isolated complete intersection singularity (ICIS), defined by an analytic map $f=(f_1,\ldots,f_k):(\C^{n+k},0) \rightarrow (\C^k,0)$. We set $N=n+k$. 
\par Let $U$ be a neighborhood of the origin in $\C^{N}$ where all the functions $f_i$ $(i=1,\ldots,k)$ and the 1-form $\omega$ are defined.
Let $S_{\delta}\subset U$ be a sufficiently small sphere around the origin which intersects $V$ transversely. Let $K=V\cap S_{\delta}$ be the link of $(V,0)$. 
The set $\{\omega(z), df_1(z), . . . , df_k(z)\}$ is a $(k+ 1)$-frame in the space
dual to $\C^{n+k}$ for all $z\in K$. Therefore, one has a map
\[\Psi = (\omega, df_1, \ldots, df_k) : K\to W_{N ,k+1},\]
 where $W_{N ,k+1}$ is the Stiefel manifold of $(k + 1)$-frames in the space dual to
$\C^{N}$. 
\begin{definition}
The  index $\operatorname{Ind}_{\operatorname{GSV} }(\omega;V,0)$ of the 1-form $\omega$ on the ICIS $V$
at the origin is defined as the degree of the map
\[\Psi = (\omega, df_1, \ldots , df_k) : K\to W_{N ,k+1}.\]
\end{definition}
Note that $\operatorname{Ind}_{\operatorname{GSV} }(\omega;V,0)$ is finite if and only if the 1-forms $\omega,df_{1},\ldots,df_k$ are linearly independent for all points over $K$.
\par According to \cite{Ebeling2001} or \cite[Theorem 5.3.35]{Ebeling2023}, the index $\operatorname{Ind}_{\operatorname{GSV} }(\omega;V,0)$ can also be defined by
\begin{align*}
\operatorname{Ind}_{\operatorname{GSV} }(\omega;V,0):=\dim_{\C} \frac{\mathcal{O}_{\C^N,0}}{I},
\end{align*}
where $I$ is the ideal generated by $f_1,\ldots,f_k$ and the $(k+1) \times (k+1)$-minors of the matrix
\begin{align*}
\left( \begin{array}{ccc}
\displaystyle\frac{\partial f_1}{\partial x_1} & \cdots & \displaystyle\frac{\partial f_1}{\partial x_{N}} \\
\vdots & \cdots & \vdots \\
\displaystyle\frac{\partial f_k}{\partial x_1} & \cdots & \displaystyle\frac{\partial f_k}{\partial x_{N}} \\
A_1 & \cdots & A_{N}
\end{array} \right).
\end{align*}

For example, when $k=1, \ n=1$ and $N=2$, if $\omega=A dx+B dy$ and $V=X=\{ \phi=0 \}$, we have that $I$ is generated by $\phi$ and the $2 \times 2$ minors of 
\begin{align*}
\left( \begin{array}{cc}
\displaystyle\frac{\partial \phi}{\partial x} & \displaystyle\frac{\partial \phi}{\partial y} \\
A & B
\end{array} \right),
\end{align*}
i.e., $I=\langle \phi, \frac{\partial \phi}{\partial x} B- \frac{\partial \phi}{\partial y} A \rangle$. Then, we have
\begin{align}\label{tangency}
\operatorname{Ind}_{\operatorname{GSV} }(\omega;X,0)=\dim_{\C} \frac{\mathcal{O}_{\C^2,0}}{\langle \phi, \frac{\partial \phi}{\partial x} B- \frac{\partial \phi}{\partial y} A \rangle}.
\end{align}
\par Let $X$ be an isolated hypersurface singularity defined by
$\phi:(\C^n,0)\to(\C,0)$. If $\displaystyle\omega=\sum_{j=1}^{n}A_j dz_j$, then we can write 
\begin{equation}\label{index}
\operatorname{Ind}_{\operatorname{GSV} }(\omega;X,0)=\dim_{\mathbb{C}}\dfrac{\mathcal{O}_n}{\langle\phi,\frac{\partial\phi}{\partial x_j}A_k-\frac{\partial\phi}{\partial x_k}A_j\rangle_{(j,k)\in \Lambda}}
\end{equation}
where $\Lambda=\{(j,k);\;j,k=1,\ldots,n,\;j\neq k\}$.
 According to L\^e-Greuel formula (\cite{Greuel} or \cite{Le1974}), we have
\begin{equation}\label{Le}
   \operatorname{Ind}_{\operatorname{GSV} }(df;X,0)=\mu_0(X)+\mu_0(\phi,f),
\end{equation}
where $f:(\C^n,0)\to(\C,0)$ is a holomorphic function with an isolated singularity on $(X,0)$.
\section{Radial index and the Euler Obstruction}\label{Euler}
In this section, we give some definitions and known results about the radial index and the local Euler obstruction of a 1-form. For a more detailed account of a general theory of indices of vector fields and 1-forms, we refer the reader to \cite{Ebeling2023}. We will use these results in the case of complex hypersurfaces with an isolated singularity. 
\par Let $V$ be a closed (real) subanalytic variety in a smooth manifold $M$, where $M$ is
equipped with a (smooth) Riemannian metric. Let 
$V=\cup_{i=1}^{q} V_i$ be a subanalytic
Whitney stratification of $V$ (see \cite{Trotman} for this notion). Let $\omega$ be the germ at $p\in\C^{N}$ of a (continuous) 1-form on $(V, p)$, i.e. the
restriction to $V$ of a 1-form defined in a neighborhood of the point $p$ in the ambient
manifold $M$. Let $V=\cup_{i=1}^{q} V_i$ be a subanalytic Whitney stratification of $V$. 
For each $p\in V$, set $V_{(p)} = V_i$ be the stratum containing $p$.
A point $p\in V$ is a singular point of $\omega$ if the restriction of 
$\omega$ to the stratum $V_{(p)}$ containing
$p$ vanishes at the point $p$. 
\begin{definition}
The germ $\omega$ of a 1-form at the point $p$ is called \textit{radial} if, for all $\epsilon>0$
small enough, the 1-form is positive on the outward normals to the boundary of the
$\epsilon$-neighborhood of the point $p$.
\end{definition}

Let $p\in V_i = V_{(p)}$, $\dim V_{(p)} = k$, and let $\eta$ be a 1-form defined in a
neighborhood of the point $p$. As above, let $N_i$ be a normal slice (relative to the Riemannian metric) of $M$ to the stratum $V_i$ at the point $p$ and $h$ a diffeomorphism
from a neighborhood of $p$ in $M$ to the product $U_{i}(p)\times N_i$, where $U_{i}(p)$ is an
$\epsilon$-neighborhood of $p$ in $V_i$, which is the identity on $U_{i}(p)$. 
 \par A 1-form $\eta$ is called a \textit{radial extension} of the 1-form $\eta|_{V_{(p)}}$
if there
exists such a diffeomorphism $h$ which identifies $\eta$ with the restriction to $V$ of the 1-
form $\pi_1^{*}\eta|_{V_{(p)}}+\pi_2^{*}\eta^{\operatorname{rad}}_{N_{i}}$, where $\pi_1$ and $\pi_2$ are the projections from a neighborhood of $p$ in $M$ to $V_{(p)}$ and $N_i$ respectively, and $\eta^{\operatorname{rad}}_{N_{i}}$
 is a radial 1-form on $N_i$.
\par For a 1-form $\omega$ on $(V,p)$ with an isolated singular point at the point $p$ there
exists a 1-form $\tilde{\omega}$ on $V$ such that 
 \begin{enumerate}
\item $\tilde{\omega}$ coincides with $\omega$ on a neighborhood of the intersection of $V$
with the boundary $\partial B_{\epsilon}$ of the $\epsilon$-neighborhood around the point $p$;
\item the 1-form $\tilde{\omega}$ has a finite number of singular points (zeros);
\item in a neighborhood of each singular point $q\in V\cap B_{\epsilon}$, $q \in V_i$, 
the 1-form $\tilde{\omega}$ is a radial extension of its restriction to the stratum $V_i$.
\end{enumerate}
\begin{definition}
The \textit{radial index} $\operatorname{Ind}_{\operatorname{rad}}(\omega;V,p)$ of the 1-form $\omega$ at the point $p$ is 
\[\operatorname{Ind}_{\operatorname{rad}}(\omega;V,p)=\sum_{q\in \sing(\tilde{\omega})}\operatorname{ind}(\tilde{\omega}|_{V_{(q)}};V_{(q)},q)\]
where $\operatorname{ind}(\tilde{\omega}|_{V_{(q)}};V_{(q)},q)$ is the usual index of the restriction of the 1-form $\tilde{\omega}$ to the stratum $V_{(q)}$, see \cite[Definition 5.2.3]{Ebeling2023}.
\end{definition}
The radial index $\operatorname{Ind}_{\operatorname{rad}}(\omega;V,p)$ is well-defined, see \cite[Proposition 5.3.9]{Ebeling2023}. It follows from the definition that $\operatorname{Ind}_{\operatorname{rad}}(\omega;V,p)$ satisfies the \textit{law of conservation of number}, and moreover one has a Poincar\'e-Hopf type theorem \cite[Theorem 5.3.10]{Ebeling2023}: for a compact real subanalytic variety $V$ and a 1-form $\omega$ with isolated singular points on $V$, we have
\[\sum_{q\in \sing(\omega)}\operatorname{Ind}_{\operatorname{rad}}(\omega;V,q)=\chi(V),\]
where $\chi(V)$ denotes the \textit{Euler characteristic} of the space (variety) $V$.
\par For 1-form $\omega$ one has the following proposition proved in \cite[Proposition 2.8]{Ebeling2004}.
\begin{proposition}\label{radial_gsv}
For a 1-form $\omega$ on an ICIS $(V,0)$ we have
\[\operatorname{Ind}_{\operatorname{GSV} }(\omega;V,0)=\operatorname{Ind}_{\operatorname{rad}}(\omega;V,0)+\mu_0(V),\]
where $\mu_0(V)$ denotes the Milnor number of $(V,0)$, in the sense of Hamm \cite{Hamm}. 
\end{proposition}

Now, we define the \textit{local Euler obstruction} of a singular point of a 1-form following \cite[Section 5.3.5]{Ebeling2023}. The initial idea was given by MacPherson \cite{Mac} who defined the \textit{Euler obstruction} of a singular point of a complex analytic variety. For a recent reference where the reader can find an explicit definition, we refer to \cite{Callejas}.

\par We start with the definition of the Nash transformation of a germ of a singular variety. Let $(V,0)\subset(\C^N,0)$ be the germ of a purely $n$-dimensional complex analytic variety. We assume that $V$ is a representative of $(V,0)$ defined in a suitable neighborhood $U$ of $0\in\C^N$. Let $G(n,N)$ be the Grassmann manifold of $n$-dimensional vector subspaces of $\C^N$. Let $V_{reg}$ be the non-singular part of $V$. There is a natural map $\sigma:V_{reg}\to U\times G(n,N)$ which is defined by $\sigma(z)=(z,T_{z} V_{reg})$. The \textit{Nash transform} $\widehat{V}$ is the closure of the image $\operatorname{Im}\sigma$ of the map $\sigma$ in $U\times G(n,N)$. It is a usually singular analytic variety. There is the natural base point map $\nu:\widehat{V}\to V$. Let $\widehat{V}':=\widehat{V}\setminus \nu^{-1}(V\setminus V_{reg})$. Then the restriction $\nu|_{\widehat{V}'}$ maps $\widehat{V}'$ biholomorphically to $V_{reg}$. 
\par The \textit{Nash bundle} $\widehat{T}$ over $\widehat{V}$  is the pullback of the tautological bundle on the
Grassmann manifold $G(n, N )$ under the natural projection map 
$\widehat{V}\to G(n, N )$.  It
is a vector bundle of rank $n$. There is a natural lifting of the Nash transformation to
a bundle map from the Nash bundle $\widehat{T}$ to the restriction of the tangent bundle $T\C^{N}$
of $\C^N$ to $V$. This is an isomorphism of $\widehat{T}$ and $TV_{reg}\subset T\C^{N}$ over the regular part
$V_{reg}$ of $V$.
Let $V=\bigcup_{i=1}^{q} V_{i}$ be a subanalytic Whitney stratification of $V$ and let $\omega$ be a 1-form on $U$ with an isolated singular point on $V$ at the origin. Let $\epsilon$ be small enough such that the 1-form $\omega$ has no singular points on $V\setminus\{0\}$ inside
the ball $B_{\epsilon}$. The 1-form $\omega$ gives rise to a section $\widehat{\omega}$ of the dual Nash bundle $\widehat{T}^{*}$ over the Nash transform $\widehat{V}$ without zeros outside of $\nu^{-1}(0)$.  The following definition was given in \cite{Ebeling2005}. 
\begin{definition}
The \textit{local Euler obstruction} $\operatorname{Eu}(\omega; V,0)$ of the 1-form $\omega$ on $V$ at $0\in\C^N$ is the obstruction to extend the non-zero section $\widehat{\omega}$ from the pre-image of a neighborhood of the sphere $S_{\epsilon}=\partial B_{\epsilon}$ to the pre-image of its interior. More precisely, its value (as an element of the cohomology group $H^{2n}(\nu^{-1}(V\cap B_{\epsilon}),\nu^{-1}(V\cap S_{\epsilon}),\mathbb{Z})$) on the fundamental class of the pair $[(\nu^{-1}(V\cap B_{\epsilon}),\nu^{-1}(V\cap S_{\epsilon}))]$. 
\end{definition}
\par The Euler obstruction of a 1-form can be considered as an index. In particular, it satisfies the law of conservation of number (just as the radial index). Moreover, on a smooth variety the Euler obstruction and the radial index coincide. 
We set $\bar{\chi}(Z):=\chi(Z)-1$ and call it the \textit{reduced} (modulo a point) Euler characteristic of the topological space $Z$ (though, strictly speaking, this name is only correct for a non-empty space $Z$). 
We have the following result from \cite[Proposition 5.3.32]{Ebeling2023}.
\begin{proposition}\label{euler_radial}
Let $(V,0)\subset(\C^N,0)$ have an isolated singularity at $0\in\C^N$ of dimension $m$ and let $\ell:\C^N\to \C$ be a generic linear function. Then
\[\operatorname{Ind}_{\operatorname{rad}}(\omega;V,0)-\operatorname{Eu}(\omega;V,0)=\operatorname{Ind}_{\operatorname{rad}}(d\ell;V,0)=(-1)^{m-1}\bar{\chi}(M_{\ell}),\]
where $M_{\ell}$ is the Milnor fiber of the linear function $\ell$ on $V$. In particular, if $f$ is a germ of holomorphic function with an isolated critical point on $(V,0)$, then 
\[\operatorname{Eu}(df;V,0)=(-1)^m(\chi(M_{\ell})-\chi(M_f)),\]
where $M_f$ is the Milnor fiber of $f$.
\end{proposition}

\section{Proof of Theorem \ref{mainresult}}
First, we state a technical lemma 
inspired by \cite[Lemma 6.1]{FPGBSM}:

\begin{lemma} \label{newlema61}
Let $f_1, \ldots, f_m,g,p_1,\ldots, p_n \in \mathcal{O}_n$, where $f_i$ and $g$ are relatively prime, for any $i \in \{ 1,\ldots,m \}$. Assume that $g$ is a non-zero-divisor on $\mathcal{O}_n/\langle f_1,\ldots,f_m\rangle$ and the terms on the right are finite. Then
\begin{eqnarray*}
\dim_{\C} \frac{\mathcal{O}_n}{\langle f_1,\ldots,f_m,gp_1,\ldots,gp_n\rangle}&=&
\dim_{\C} \frac{\mathcal{O}_n}{\langle f_1,\ldots,f_m,p_1,\ldots,p_n\rangle}\\
& & +\dim_{\C} \frac{\mathcal{O}_n}{\langle f_1,\ldots,f_m,g\rangle}.
\end{eqnarray*}
\end{lemma}

\begin{proof}
Observe that \[\displaystyle\dim_{\C} \frac{\mathcal{O}_n}{\langle f_1,\ldots,f_m,r_1,\ldots,r_n\rangle}=\dim_{\C} \frac{\mathcal{O}}{\langle r_1',\ldots,r_n'\rangle},\] with $\mathcal{O}=\displaystyle\frac{\mathcal{O}_n}{\langle f_1,\ldots,f_m\rangle}$ and $r_i'=r_i+\langle f_1,\ldots,f_m\rangle$, for any $i \in \{ 1,\ldots,n \}$ and any $r_i \in \mathcal{O}_n$. Since $g$ is a non-zero-divisor on $\mathcal{O}$, the proof is completed using the following exact sequence:
\begin{align*}
0 \longrightarrow \frac{\mathcal{O}}{\langle p_1',\ldots,p_n'\rangle} \xrightarrow{ \ \sigma  \ } \ \frac{\mathcal{O}}{\langle g'p_1',\ldots,g'p_n'\rangle} \ \xrightarrow{ \ \delta  \ } \frac{\mathcal{O}}{\langle g'\rangle} \longrightarrow 0,
\end{align*}
where $\sigma(z'+\langle p_1',\ldots,p_n'\rangle)=g'z'+\langle g'p_1',\ldots,g'p_n'\rangle$ and $\delta(z'+\langle g'p_1',\ldots,g'p_n'\rangle)=z'+\langle g'\rangle$ for any $z' \in \mathcal{O}$.
\end{proof}
\begin{remark}\label{theta}
    We consider $\Theta_X^T$ the $\mathcal{O}_n$-submodule of $\Theta_X$ generated by the trivial vector fields, that is, $\Theta_X^T$ is generated by
\[
\phi\dfrac{\partial}{\partial x_i},\; \dfrac{\partial\phi}{\partial x_j}\dfrac{\partial}{\partial x_k}-\dfrac{\partial\phi}{\partial x_k}\dfrac{\partial}{\partial x_j},\;\;\mbox{with }i,j,k = 1,\ldots, n;\; k \neq j,
\]
where $\phi:(\mathbb{C}^n,0)\to (\mathbb{C},0)$ is the defining equation of $(X,0)$. 
\end{remark} 

\begin{lemma}\label{exact1}
Let $X=\{\phi=0\}\subset(\mathbb{C}^n,0)$, $n\geq2$, be a hypersurface germ with an isolated singularity, and let $\omega$ be a germ of a holomorphic $1$-form with an isolated singularity at the origin. Assume that $\mu_{BR}(\omega,X)<\infty$. Then the following sequence of $\mathbb{C}$-vector spaces is exact:

\[ 0\longrightarrow \frac{\Theta_X}{\Theta_X^T} \xrightarrow{\cdot\omega} \frac{\mathcal{O}_n}{\omega(\Theta_X^T)} \xrightarrow{\beta} \frac{\mathcal{O}_n}{\omega(\Theta_X)} \longrightarrow0, \]
where the maps are defined by

\[ (\cdot\omega)(\xi+\Theta_X^T)=\omega(\xi)+\omega(\Theta_X^T), \qquad \beta(h+\omega(\Theta_X^T))=h+\omega(\Theta_X). \]

\end{lemma}

\begin{proof}
Both maps are well defined. Moreover, $\beta$ is surjective and
\[ \operatorname{Im}(\cdot\omega)=\frac{\omega(\Theta_X)}{\omega(\Theta_X^T)}=\ker\beta. \]
It remains to prove that $\cdot\omega$ is injective. Let $\xi\in\Theta_X$ satisfy $\omega(\xi)\in\omega(\Theta_X^T)$. Choose $\eta\in\Theta_X^T$ such that $\omega(\xi)=\omega(\eta)$, and set $\zeta=\xi-\eta$. Then $\zeta\in\Theta_X$ and $\omega(\zeta)=0$. Write $\omega=\sum_{i=1}^n A_i\,dx_i$ and set $\phi_i=\partial\phi/\partial x_i$.

Since $\omega$ has an isolated singularity, $A_1,\ldots,A_n$ form a regular sequence. Hence their syzygies are generated by the Koszul relations, and there exist $h_{ij}\in\mathcal{O}_n$ such that
\[ \zeta=\sum_{i<j}h_{ij}\left(A_i\frac{\partial}{\partial x_j}-A_j\frac{\partial}{\partial x_i}\right). \]
Consequently,
\[ d\phi(\zeta)\in J:=\langle A_i\phi_j-A_j\phi_i:1\leq i<j\leq n\rangle. \]

On the other hand, $\zeta\in\Theta_X$ implies that $d\phi(\zeta)=\lambda\phi$ for some $\lambda\in\mathcal{O}_n$. Thus $\lambda\phi\in J$.

We claim that $(J:\phi)=J$. Indeed, the finiteness of $\mu_{BR}(\omega,X)$ implies that $J+(\phi)$ has finite colength. Since $J$ is the ideal of maximal minors of the matrix
\[ \begin{pmatrix} A_1&\cdots&A_n\\ \phi_1&\cdots&\phi_n \end{pmatrix}, \]
the determinantal height bound gives $\operatorname{ht}J\leq n-1$. The finite colength of $J+(\phi)$ gives the reverse inequality, so $\operatorname{ht}J=n-1$. By the Eagon--Northcott theorem~\cite[Theorem~2.16(c)]{BrunsVetter}, $\mathcal{O}_n/J$ is therefore Cohen--Macaulay of dimension one. Since the class of $\phi$ is a parameter in this quotient, it is a non-zero-divisor, proving the claim.

It follows that $\lambda\in J\subset J_\phi$, where $J_\phi=\langle\phi_1,\ldots,\phi_n\rangle$. Write $\lambda=\sum_i b_i\phi_i$. Then
\[ d\phi\left(\zeta-\phi\sum_i b_i\frac{\partial}{\partial x_i}\right)=0. \]
Since $\phi$ has an isolated singularity, its partial derivatives form a regular sequence. Thus $\ker d\phi$ is generated by the Koszul vector fields
\[ \phi_i\frac{\partial}{\partial x_j}-\phi_j\frac{\partial}{\partial x_i}, \qquad 1\leq i<j\leq n, \]
which belong to $\Theta_X^T$. As $\phi\Theta_n\subset\Theta_X^T$, we obtain $\zeta\in\Theta_X^T$. Finally, $\xi=\zeta+\eta\in\Theta_X^T$. Therefore $\cdot\omega$ is injective, completing the proof.
\end{proof}

\subsection{Proof of Theorem \ref{mainresult}}
From Lemma \ref{exact1}, we obtain
\begin{equation}\label{Eq1}\mu_{BR}(\omega,X)=\dim_{\mathbb{C}}\dfrac{\mathcal{O}_n}{\omega(\Theta_X^T)}-\dim_{\mathbb{C}}\dfrac{\Theta_X}{\Theta_X^T}.\end{equation}
Note that, by Remark \ref{theta} and Lemma \ref{newlema61}, we have
\[\begin{array}{rl}\dim_{\mathbb{C}}\dfrac{\mathcal{O}_n}{\omega(\Theta_X^T)}&=\dim_{\mathbb{C}}\dfrac{\mathcal{O}_n}{\langle\phi A_1,\ldots,\phi A_n,\frac{\partial\phi}{\partial x_j}A_k-\frac{\partial\phi}{\partial x_k}A_j\rangle_{(j,k)\in \Lambda}}\\&\\&=\dim_{\mathbb{C}}\dfrac{\mathcal{O}_n}{\langle\phi,\frac{\partial\phi}{\partial x_j}A_k-\frac{\partial\phi}{\partial x_k}A_j\rangle_{(j,k)\in\Lambda}}\\&\\&\;\;+\dim_{\mathbb{C}}\dfrac{\mathcal{O}_n}{\langle A_1,\ldots,A_n,\frac{\partial\phi}{\partial x_j}A_k-\frac{\partial\phi}{\partial x_k}A_j\rangle_{(j,k)\in \Lambda}},
\end{array}\]
where $\Lambda = \{(j, k): j, k = 1,\ldots, n,\,\,\, j\neq k\}$. 
Since
\[\frac{\partial\phi}{\partial x_j}A_k-\frac{\partial\phi}{\partial x_k}A_j\in \langle A_1,\ldots, A_n\rangle,\mbox{ for all }(j,k)\in \Lambda,\]
we get
\[\dim_{\mathbb{C}}\dfrac{\mathcal{O}_n}{\langle A_1,\ldots,A_n,\frac{\partial\phi}{\partial x_j}A_k-\frac{\partial\phi}{\partial x_k}A_j\rangle_{(j,k)\in \Lambda}}=\dim_{\mathbb{C}}\dfrac{\mathcal{O}_n}{\langle A_1,\ldots,A_n\rangle}=\mu_0(\omega).\]
Thus, we deduce
\begin{equation}\label{Eq2}
\dim_{\mathbb{C}}\dfrac{\mathcal{O}_n}{\omega(\Theta_X^T)}= \dim_{\mathbb{C}}\dfrac{\mathcal{O}_n}{\langle\phi,\frac{\partial\phi}{\partial x_j}A_k-\frac{\partial\phi}{\partial x_k}A_j\rangle_{(j,k)\in\Lambda}}+\mu_0(\omega).   
\end{equation}
Finally, since $\dim_{\mathbb{C}}\dfrac{\Theta_X}{\Theta_X^T}=\tau_0(X)$ by Tajima's theorem \cite{Tajima}, the proof of Theorem \ref{mainresult} concludes by
substituting \eqref{Eq2} in \eqref{Eq1} and considering equation (\ref{index}).

As a consequence, we recover one of the main results from Nu\~{n}o-Ballesteros--Or\'efice-Okamoto--Pereira-Tomazella \cite[Corollary 4.1]{Nuno2020}. 
\begin{maincorollary}
Let $X$ be an isolated hypersurface singularity defined by
$\phi :(\mathbb{C}^n,0)\to (\mathbb{C}, 0)$ and let $f\in\mathcal{O}_n$ satisfy $\mu_{BR}(f,X)<\infty$. Then,
\[\mu_{BR}(f,X)=\mu_0( f )+\mu_0(\phi, f ) + \mu_0(X) -\tau_0(X).\]
\end{maincorollary}
\begin{proof} 
Indeed, consider the 1-form $\omega$ given by $\omega= df$. By Theorem \ref{mainresult}, we have:
\begin{equation}\label{eq_6}
\mu_{BR}(f,X)
= \mu_{BR}(df,X)=\operatorname{Ind}_{\operatorname{GSV} }(df;X,0)+\mu_0(df)-\tau_0(X).
\end{equation}
Since $\mu_0(f) =\mu_0(df)$, the proof concludes by applying formula (\ref{Le})
\[\operatorname{Ind}_{\operatorname{GSV} }(df;X,0)=\mu_0(X)+\mu_0(\phi,f)\]
and substituting this into equation (\ref{eq_6}).
\end{proof}
As a second corollary, we establish a connection between the Bruce--Roberts number $\mu_{BR}(\omega,X)$ and other indices of 1-forms along $X$.
\begin{secondcorollary}
 Let $\omega$ be a germ of a holomorphic 1-form with isolated singularity at $0\in\mathbb{C}^n$, $n\geq 2$. Let $X$ be a germ of a complex analytic hypersurface with an isolated singularity at $0\in\C^n$.
 Assume that $\mu_{BR}(\omega,X)<\infty$. Then
 \[\mu_{BR}(\omega,X)-\mu_0(\omega)+\tau_0(X)-\mu_0(X)-\operatorname{Eu}(\omega;X,0)=(-1)^{n-2}\bar{\chi}(M_{\ell}),\]
 where $M_{\ell}$ is the Milnor fiber of the generic linear function $\ell:\mathbb{C}^n\to\mathbb{C}$ on $X$ and $\bar{\chi}(M_{\ell})=\chi(M_{\ell})-1$. 
\end{secondcorollary}
\begin{proof}
We have $\mu_{BR}(\omega,X)=\operatorname{Ind}_{\operatorname{GSV} }(\omega;X,0)+\mu_0(\omega)-\tau_0(X)$ by Theorem \ref{mainresult}. On the other hand, $\operatorname{Ind}_{\operatorname{GSV} }(\omega;X,0)=\operatorname{Ind}_{\operatorname{rad} }(\omega;X,0)+\mu_0(X)$ by Proposition \ref{radial_gsv}. Therefore, 
\begin{equation}\label{eq_5}
\mu_{BR}(\omega,X)=\operatorname{Ind}_{\operatorname{rad} }(\omega;X,0)+\mu_0(X)+\mu_0(\omega)-\tau_0(X).
\end{equation}
The proof concludes by applying Proposition \ref{euler_radial} to $(X,0)$ and substituting into equation (\ref{eq_5}). 
\end{proof}
We finish this section with an example where Theorem \ref{mainresult} is verified. 
\begin{example} \label{ex4}
Let $X=\{\phi=0\}$, where $\phi:(\C^3,0) \rightarrow (\C,0)$ is given by $\phi(x,y,z)=x^3+yz^2+y^3+xy^4$. Set $E=(4x^2y+3x)\frac{\partial}{\partial x}+(2xy^2+3y)\frac{\partial}{\partial y}+(5xyz+3z)\frac{\partial}{\partial z}$. Since $E(\phi)=(9+12xy)\phi$ and $9+12xy$ is a unit, the module $\Theta_X$ is generated by (see \cite[Example 9]{Tajima2021}):
\begin{align*}
    \Theta_X= \left\langle E,\frac{\partial{\phi}}{\partial y} \frac{\partial}{\partial x}-\frac{\partial\phi}{\partial x}\frac{\partial}{\partial y},\frac{\partial{\phi}}{\partial z} \frac{\partial}{\partial x}-\frac{\partial\phi}{\partial x}\frac{\partial}{\partial z},\frac{\partial{\phi}}{\partial z} \frac{\partial}{\partial y}-\frac{\partial\phi}{\partial y}\frac{\partial}{\partial z}\right\rangle
\end{align*}
Let $\omega=zdx+xdy+ydz$. 
Using \textit{Singular} (\cite{DGPS}) to compute the indices, we verify that Theorem \ref{mainresult} holds, since $\mu_{BR}(\omega,X)=14$, $\operatorname{Ind}_{\operatorname{GSV} }(\omega;X,0)=21, \ \mu_0(\omega)=1$ and $\tau_0(X)=8$.
\end{example}

\section{Relative Bruce--Roberts number for 1-forms}
Let $\omega$ be a germ of a holomorphic 1-form with isolated singularity at $(\mathbb{C}^n,0)$ and let $X$ be a complex analytic germ with isolated singularity at $(\mathbb{C}^n,0)$. We define the \textit{relative Bruce--Roberts number} of the 1-form $\omega$ relative to $X$ as
\[
\mu_{BR}^{-}(\omega,X)=\dim_{\C}\frac{\mathcal{O}_n}{\omega(\Theta_X)+I_X}.\]
\par To establish the relationship between $\mu_{BR}(\omega,X)$ and $\mu_{BR}^-(\omega,X)$, we state the following Lemma which is analogous to \cite[Lemma 2.2]{Lima2021} and whose proof follows the same approach. Given a matrix $A$ with entries in a ring $R$, we denote by $I_k(A)$ the ideal in $R$ generated by the $k\times k$ minors of $A$.
\begin{lemma}\label{lima_lema}
Let $\omega=\displaystyle\sum^n_{i=1}A_idx_i\in\Omega^1(\mathbb{C}^n,0)$ and let $g \in \mathcal{O}_n $ be such that $\dim V((\omega,dg)) = 1$ and $V(\omega) = \{0\}$ (here $(\omega):=\langle A_1,\ldots,A_n\rangle$ and $(\omega,dg)$  is the
ideal in $\mathcal{O}_n$ generated by the maximal minors of the matrix of $\left(A_i,\frac{\partial g}{\partial x_i}\right)$). Consider the following matrices
\[
A = \left(\begin{array}{ccc}
A_1 & \cdots & A_n\\
\frac{\partial g}{\partial x_1} & \cdots & \frac{\partial g}{\partial x_n}
\end{array}\right),
\quad
A' = \left(\begin{array}{cccc}
\mu & A_1 & \cdots & A_n \\
\lambda & \frac{\partial g}{\partial x_1} & \cdots &  \frac{\partial g}{\partial x_n}
\end{array}\right),
\]
where $\lambda, \mu \in \mathcal{O}_n$. Let $M, M'$ be the submodules of $\mathcal{O}^2_n$ generated by the columns of $A, A'$ respectively. If $I_2(A) = I_2(A')$ then $M = M'$.
\end{lemma}
We can now prove a theorem similar to \cite[Theorem 2.3]{Lima2021}.

\begin{theorem}\label{relthmBR} Let $\omega=\displaystyle\sum^n_{i=1}A_i(x)dx_i$ be a germ of a holomorphic 1-form with isolated singularity at $0\in\mathbb{C}^n$, $n\geq 2$, where $A_i\in\mathcal{O}_n$. Let $X=\{\phi=0\}$ be a germ of a complex analytic hypersurface with an isolated singularity at $0\in\C^n$.
 Assume that $\mu_{BR}(\omega,X)<\infty$. Then
 \begin{itemize}
     \item[(i)]$\dfrac{\Theta_X}{\Theta_X^T}\cong\dfrac{\omega(\Theta_X)+I_X}{\omega(\Theta_X^T)+I_X}$;
     \item[(ii)] $\omega(\Theta_X)\cap I_X=(\omega) I_X$;
     \item[(iii)] $\dfrac{\mathcal{O}_n}{(\omega)}\cong\dfrac{\omega(\Theta_X)+I_X}{\omega(\Theta_X)}$,
 \end{itemize}
where $(\omega):=\langle A_1,\ldots,A_n\rangle$ and $I_X=\langle\phi\rangle$.
\end{theorem}
\begin{proof} 
Write \(\phi_i=\partial\phi/\partial x_i\), and set
\[
A=
\begin{pmatrix}
A_1&\cdots&A_n\\
\phi_1&\cdots&\phi_n
\end{pmatrix},
\qquad
(\omega,d\phi):=I_2(A).
\]
The finiteness assumption implies that \(\dim V((\omega,d\phi))=1\) and that \(\phi\) is regular in \(\mathcal O_n/(\omega,d\phi)\).

(i) The homomorphism
\[
\Psi:\Theta_X\longrightarrow\omega(\Theta_X)+I_X,
\qquad
\Psi(\xi)=\omega(\xi),
\]
induces the epimorphism
\[
\overline{\Psi}:\Theta_X\longrightarrow
\frac{\omega(\Theta_X)+I_X}
{\omega(\Theta_X^T)+I_X}.
\]
Clearly, \(\Theta_X^T\subseteq\ker\overline{\Psi}\). To prove the reverse inclusion, let \(\xi\in\ker\overline{\Psi}\). Then there exist \(\eta\in\Theta_X^T\) and \(\mu,\lambda\in\mathcal O_n\) such that
\[
\omega(\xi-\eta)=\mu\phi,
\qquad
d\phi(\xi-\eta)=\lambda\phi.
\]
Consequently,
\[
\begin{pmatrix}
\mu\phi\\
\lambda\phi
\end{pmatrix}
\in
\left\langle
\begin{pmatrix}
A_i\\
\phi_i
\end{pmatrix}
:\ i=1,\ldots,n
\right\rangle,
\]
and therefore
\[
I_2
\begin{pmatrix}
\mu\phi&A_1&\cdots&A_n\\
\lambda\phi&\phi_1&\cdots&\phi_n
\end{pmatrix}
=
I_2(A).
\]
In particular,
\[
\phi\det
\begin{pmatrix}
\mu&A_i\\
\lambda&\phi_i
\end{pmatrix}
\in(\omega,d\phi),
\qquad i=1,\ldots,n.
\]
Since \(\phi\) is regular modulo \((\omega,d\phi)\), we obtain
\[
\det
\begin{pmatrix}
\mu&A_i\\
\lambda&\phi_i
\end{pmatrix}
\in(\omega,d\phi),
\qquad i=1,\ldots,n.
\]
Thus
\[
I_2
\begin{pmatrix}
\mu&A_1&\cdots&A_n\\
\lambda&\phi_1&\cdots&\phi_n
\end{pmatrix}
=
I_2(A).
\]

By Lemma \ref{lima_lema}, applied with \(g=\phi\),
\[
\begin{pmatrix}
\mu\\
\lambda
\end{pmatrix}
\in
\left\langle
\begin{pmatrix}
A_i\\
\phi_i
\end{pmatrix}
:\ i=1,\ldots,n
\right\rangle.
\]
In particular, \(\mu\in(\omega)\) and \(\lambda\in J_\phi\). Write
\[
\lambda=\sum_{i=1}^n b_i\phi_i,
\qquad b_i\in\mathcal O_n,
\]
and set \(\zeta=\xi-\eta\). Then
\[
d\phi\left(
\zeta-\phi\sum_{i=1}^n
b_i\frac{\partial}{\partial x_i}
\right)=0.
\]
Since \(\phi\) has an isolated singularity, \(\ker d\phi\subseteq\Theta_X^T\). Moreover, \(\phi\,\partial/\partial x_i\in\Theta_X^T\) for every \(i\). Hence \(\zeta\in\Theta_X^T\), and consequently
\[
\xi=\zeta+\eta\in\Theta_X^T.
\]
Therefore, \(\ker\overline{\Psi}=\Theta_X^T\), and the first isomorphism theorem proves (i).

(ii) Let \(h\in\omega(\Theta_X)\cap I_X\). There exist \(\eta\in\Theta_X\) and \(\mu,\lambda\in\mathcal O_n\) such that
\[
h=\omega(\eta)=\mu\phi,
\qquad
d\phi(\eta)=\lambda\phi.
\]
Using the same argument as in (i) and applying Lemma `\ref{lima_lema}`, we obtain \(\mu\in(\omega)\). Thus
\[
h=\mu\phi\in(\omega)I_X.
\]

Conversely, for each \(i\),
\[
\phi A_i
=
\omega\left(
\phi\frac{\partial}{\partial x_i}
\right)
\in\omega(\Theta_X)\cap I_X.
\]
Since these elements generate \((\omega)I_X\), the reverse inclusion follows.

(iii) By the second isomorphism theorem and item (ii),
\[
\begin{aligned}
\frac{\omega(\Theta_X)+I_X}{\omega(\Theta_X)}
&\cong
\frac{I_X}{\omega(\Theta_X)\cap I_X}\\
&=
\frac{I_X}{(\omega)I_X}\\
&\cong
\frac{\mathcal O_n}{(\omega)}.
\end{aligned}
\]
The last isomorphism is induced by multiplication by \(\phi\). \end{proof}

\begin{remark} 
If $X=\{\phi=0\}$ is a germ of a complex analytic hypersurface with an isolated singularity at $0\in\C^n$, by \cite[Theorem 2.3]{Tajima} or \cite[Corollary 3.4]{Nuno2020}, then \[\dim_{\mathbb{C}} \dfrac{\Theta_X}{\Theta_X^T} = \tau_0(X).\]
Therefore, by Theorem \ref{relthmBR} item $(i)$, we have
\[ \dim_{\mathbb{C}} \dfrac{\omega(\Theta_X) +I_X}{\omega(\Theta_X^T) +I_X} = \tau_0(X). \]
\end{remark}

\subsection{Proof of Theorem \ref{correlthmBr}}
From Theorem \ref{relthmBR} item $(iii)$, we have
\[\mu_0(\omega)=\dim_{\mathbb{C}}\dfrac{\mathcal{O}_n}{(\omega)}=\dim_{\mathbb{C}}\dfrac{\omega(\Theta_X)+I_X}{\omega(\Theta_X)}.\]
   Hence, by the exact sequence
\[0\longrightarrow\dfrac{\omega(\Theta_X)+I_X}{\omega(\Theta_X)}\longrightarrow\dfrac{\mathcal{O}_n}{\omega(\Theta_X)}\longrightarrow\dfrac{\mathcal{O}_n}{\omega(\Theta_X)+I_X}\longrightarrow0,\]
we conclude the proof of the theorem.

\begin{remark} 
If we take $\omega=df$ in the Theorem \ref{correlthmBr}, we recover the formula
\[\mu_{BR}(f,X)=\mu_0(f)+\mu^-_{BR}(f,X)\]
which is found in Section 3 of \cite{Lima2021}.
\end{remark}

\begin{corollary}\label{relmainthm}
Let $\omega$ be a germ of a holomorphic 1-form with isolated singularity at $0\in\mathbb{C}^n$, $n\geq 2$. Let $X$ be a germ of a complex analytic hypersurface with an isolated singularity at $0\in\C^n$.
 Assume that $\mu_{BR}(\omega,X)<\infty$. Then
    \[\mu_{BR}^-(\omega,X)=\operatorname{Ind}_{\operatorname{GSV} }(\omega;X,0)-\tau_0(X).\]    
\end{corollary}

\begin{proof} The proof follows from Theorems \ref{mainresult} and \ref{correlthmBr}.    
\end{proof}
    \begin{corollary}[Theorem 2.5 in \cite{Lima2021}] 
    Let $X$ be a germ of a complex analytic hypersurface with an isolated singularity at $0\in\C^n$ defined by $\phi:(\mathbb{C}^n,0)\to (\mathbb{C}, 0)$ and $f\in\mathcal{O}_n$ be a function germ such that $\mu_{BR}(f, X) <\infty$. Then $(\phi, f)$ defines an ICIS and
\[\mu_0(\phi,f) =\mu_{BR}^-(f,X) + \tau_0(X) - \mu_0(X).\]
        
    \end{corollary}
    \begin{proof} Consider the 1-form $\omega$ given by $\omega= df$. By Corollary \ref{relmainthm}, we have
\[\mu_{BR}^-(df,X)=\operatorname{Ind}_{\operatorname{GSV} }(df;X,0)-\tau_0(X)\]
The proof follows by applying the L\^e-Greuel formula (see equation \eqref{Le}): 
\[\operatorname{Ind}_{\operatorname{GSV} }(df;X,0)=\mu_0(X)+\mu_0(\phi,f).\]
    \end{proof}
    
\section{The Bruce--Roberts number for foliations on $(\C^2,0)$}\label{foliations}
In this section, we study the Bruce--Roberts number for germs of holomorphic foliations on $(\C^2,0)$. A holomorphic foliation $\F$ is defined by a holomorphic 1-form
\begin{equation*}
\label{vectorfield}
\omega=A(x,y)dx+B(x,y)dy,
\end{equation*}
or by its dual vector field
\begin{equation}
\label{oneform}
v = -B(x,y)\frac{\partial}{\partial{x}} + A(x,y)\frac{\partial}{\partial{y}},
\end{equation}
where  $A, B   \in {\mathbb C}\{x,y\}$ are relatively prime. The \textit{Milnor number $\mu_0(\F)$ of $\F$ at $0\in\C^2$} is defined as the Milnor number of the 1-form $\omega=A(x,y)dx+B(x,y)dy$, i.e., 
\[\mu_0(\F)=\mu_0(\omega)=\dim_{\C}\frac{\mathcal{O}_2}{\langle A,B\rangle}.\]
\par It follows from the definition that $\mu_0(\F)=i_0(A,B)$, where $i_0(\cdot,\cdot)$ denotes the intersection between two germs of holomorphic functions. The algebraic multiplicity of $\F$ is 
$\nu(\F):=\min\{\ord(A),\ord(B)\}$.
\par Let $f(x,y)\in  \mathbb{C}[[x,y]]$. We say that $C: f(x,y)=0$  is {\em invariant} by $\F$ if $$\omega \wedge d f=(f\cdot h) dx \wedge dy,$$ for some  $h\in \mathbb{C}[[x,y]]$. If $C$ is irreducible, then we will say that $C$ is a {\em separatrix} of $\F$. The separatrix $C$ is analytic if $f$ is convergent.  We denote by $Sep_0(\F)$ the set of all separatrices of $\F$. When $Sep_0(\F)$ is a finite set, we will say that the foliation $\F$ is {\em non-dicritical}.  Otherwise, we will say that $\F$ is {\em dicritical}.
\par Let $X$ be a reduce curve and assume that none of the irreducible components of $X$ is invariant by $\F$.
Following \cite[Chapter 2, Section 2]{Brunella-book}, we can consider the {\it tangency order of $\F$ to $X$ at $0\in\C^2$}:
\begin{align*}
\operatorname{tang}(\F,X,0)=\dim_{\C} \frac{\mathcal{O}_2}{\langle \phi,v(\phi) \rangle},
\end{align*}
where $\{ \phi=0 \}$ is the local (reduced) equation of $X$ around $0$, and $v$ is a local holomorphic vector field generating $\F$ around $0$.
\par In dimension two, Theorem \ref{mainresult} can be established as the following corollary:
\begin{corollary} \label{dim2}
Let $\F$ be a germ of a singular foliation at $(\C^2,0)$, and let $X$ be a germ of a reduced curve at $(\C^2,0)$. Assume that none of the irreducible components of $X$ is invariant by $\F$.
Then 
\begin{equation} \label{princ}
\mu_{BR}(\F,X)=\mu_0(\F)+\operatorname{tang}(\F,X,0)-\tau_0(X).
\end{equation}
\end{corollary}
\begin{proof}
Consider the foliation $\mathcal{F}$ defined by the 1-form $\omega=A\,dx+B\,dy$, where $A, B\in\mathcal{O}_2$. Let $X$ be the complex analytic curve defined by the function $\phi:(\mathbb{C}^2,0)\to (\mathbb{C},0)$. 
From equation (\ref{tangency}) it follows that 
\[\operatorname{Ind}_{\operatorname{GSV}}(\omega;X,0)=\dim_{\mathbb{C}}\dfrac{\mathcal{O}_2}{\left\langle\phi,B\frac{\partial\phi}{\partial x}-A\frac{\partial\phi}{\partial y}\right\rangle}.\]
The vector field $v$ generating $\mathcal{F}$ is given by $v=-B\frac{\partial}{\partial x}+A\frac{\partial}{\partial y}$. Therefore, by definition,
\[\operatorname{tang}(\mathcal{F},X,0)=\dim_{\mathbb{C}}\dfrac{\mathcal{O}_2}{\langle \phi, v(\phi)\rangle}=\dim_{\mathbb{C}}\dfrac{\mathcal{O}_2}{\left\langle \phi, -B\frac{\partial \phi}{\partial x}+A\frac{\partial\phi}{\partial y}\right\rangle}=\operatorname{Ind}_{\operatorname{GSV} }(\omega;X,0).\]
Hence,
\[\mu_{BR}(\mathcal{F},X)=\operatorname{tang}(\mathcal{F},X,0)+\mu_0(\mathcal{F})-\tau_0(X),\]
by Theorem \ref{mainresult}.
\end{proof}

To proceed, we present some illustrative examples:
\begin{example} \label{ex1}
Consider the curve $X$ given by $X=\{ \phi=y^p-x^q=0 \}$, and let $\F$ be a foliation defined by the $1$-form $\omega=\lambda xdy+ydx$, with $\lambda \notin \{0,-\frac{p}{q}\}$. Then, by \cite[Example 1]{Saito}, we have:
\begin{align*}
\Theta_X=\left\langle qy \frac{\partial}{\partial y}+px \frac{\partial}{\partial x}, py^{p-1} \frac{\partial}{\partial x}+qx^{q-1} \frac{\partial}{\partial y} \right\rangle.
\end{align*}
Note that $X$ is not invariant by $\F$, since
\begin{align*}
\omega \wedge d \phi = (\lambda xdy+ydx) \wedge (py^{p-1}dy -qx^{q-1}dx) = (py^p +\lambda q x^q) \ dx \wedge dy
\end{align*}
and $\lambda \neq -\displaystyle\frac{p}{q}$. Thus, we have $\mu_0(\F)=1$, $\operatorname{tang}(\F,X,0)=pq$, and $\tau_0(X)=(p-1)(q-1)$. Computing the Bruce--Roberts number in this case, we find
\begin{align*}
\mu_{BR}(\F,X)=\dim_{\C} \frac{\mathcal{O}_2}{\omega(\Theta_X)}=p+q,
\end{align*}
which indeed satisfies equation (\ref{princ}).
\end{example}

\begin{example} \label{ex3}
Consider the foliation $\F_\omega$ (Suzuki's foliation) defined by the $1$-form
\begin{align*}
\omega=(y^3+y^2-xy)dx-(2xy^2+xy-x^2)dy,
\end{align*}
and the foliation $\F_\eta$ defined by the $1$-form
\begin{align*}
\eta=(2y^2+x^3)dx-2xydy.
\end{align*}
The foliations $\F_\omega$ and $\F_\eta$ are topologically conjugate (\cite[Chapter 2, Part 3]{CervMattei}). Using the curve $X$ from Example \ref{ex1}, with $p=7$ and $q=3$, i.e., $X=\{\phi=0\}$, with $\phi=y^7-x^3$, we have $\mu_{BR}(\F_\omega,X)=\mu_{BR}(\F_\eta,X)=17$, $\mu_0(\F_\omega)=\mu_0(\F_\eta)=5$, $\operatorname{tang}(\F_\omega,X,0)=\operatorname{tang}(\F_\eta,X,0)=24$, and $\tau_0(X)=12$. Thus, for both foliations, (\ref{princ}) is satisfied.
\end{example}

\begin{proposition} \label{tangirred}
Let $\F$ be a germ of a singular foliation at $(\C^2,0)$, and let $X$ be a germ of a reduced curve at $(\C^2,0)$ defined by $\phi: (\C^2, 0) \rightarrow (\C,0)$. Suppose that $X$ is irreducible and not invariant by $\F$. Then
\begin{align*}
    \operatorname{tang}(\F,X,0)=\ord_{t=0} \ \phi^* \omega+\mu_0(X),
\end{align*}
where $\varphi$ is a Puiseux parametrization of $X$. 
\end{proposition}
\begin{proof}
Suppose that $\F$ is defined by the vector field $v = -B(x,y)\frac{\partial}{\partial{x}} + A(x,y)\frac{\partial}{\partial{y}}$. 
Then, it follows from the definition of the tangency order that
\begin{align*}
\operatorname{tang}(\F,X,0)=\text{ord}_{t=0} \ \varphi^*(A \phi_y-B\phi_x),
\end{align*}
where $\varphi(t)=(x(t),y(t))$ is a Puiseux parametrization of $X$. 
Without loss of generality, we can suppose that $y(t) \neq 0$. Since
$\phi(x(t),y(t))=0$, we obtain
\[ x'(t) \ \phi_x(x(t),y(t))+y'(t) \ \phi_y(x(t),y(t))=0 \Rightarrow \phi_y(x(t),y(t))=-\frac{x'(t) \phi_x (x(t),y(t))}{y'(t)}.
\]
Hence, 
\begin{eqnarray*}
\varphi^*(A \phi_y-B\phi_x)&=&(A \phi_y-B \phi_x)(x(t),y(t)) \\
&=& A(x(t),y(t))  \phi_y (x(t),y(t))-B(x(t),y(t))  \phi_x (x(t),y(t)) \\
& =& A(x(t),y(t)) \left( -\frac{x'(t) \phi_x (x(t),y(t))}{y'(t)} \right)-B(x(t),y(t)) \phi_x (x(t),y(t)) \\
& =&-\frac{\phi_x (x(t),y(t))}{y'(t)} \Big(x'(t) A(x(t),y(t))+y'(t)B(x(t),y(t)) \Big) \\
&=&-\frac{\phi_x (x(t),y(t))}{y'(t)} [\varphi^* \omega](t) \ =  \frac{-\varphi^*(\phi_x)(t) \cdot \varphi^* \omega(t)}{y'(t)}  \\
&=&-\frac{\varphi^*(\phi_x \cdot \omega)(t)}{y'(t)}.
\end{eqnarray*}
Thus, we get
\begin{align} \label{eqtang}
\begin{split}
\operatorname{tang}(\F,X,0) &= \text{ord}_{t=0} \ \varphi^*(A \phi_y-B\phi_x) \\
&=\text{ord}_{t=0} \ (A \phi_y-B \phi_x)(x(t),y(t)) \\
&=\text{ord}_{t=0} -\frac{\varphi^*(\phi_x \cdot \omega)(t)}{y'(t)} \\
&= \text{ord}_{t=0} \ \varphi^*(\phi_x \cdot \omega) - \text{ord}_{t=0} \ y'(t) \\
&= \text{ord}_{t=0} \ \varphi^*(\phi_x \cdot \omega) - \text{ord}_{t=0} \ y(t) +1 \\
&= \text{ord}_{t=0} \ \varphi^*\phi_x+ \text{ord}_{t=0} \ \varphi^*\omega - \text{ord}_{t=0} \ y(t) +1.
\end{split}
\end{align}
Now, from \cite[Proposition 2.1, item 3]{EveliaPloski}, we obtain
\begin{equation} \label{teissier}
    \mu_0(X)=i_0(\phi,\mathcal{P}_l(\phi))-i_0(\phi,l)+1,
\end{equation}
where $\mathcal{P}_l(\phi)=\frac{\partial \phi}{\partial x} \frac{\partial l}{\partial y}-\frac{\partial \phi}{\partial y} \frac{\partial l}{\partial x}$ is the polar of $X$ relative to $l \in \C[x,y],$ assuming that $l$ does not divide $\phi$. 
Let $l=ay-bx$ with $a$ and $b$ generic. Computing the items on the right side of ($\ref{teissier}$), we have
\begin{align*}
    i_0(\phi,\mathcal{P}_l(\phi))
    &= \text{ord}_{t=0} \ \varphi^*(a \phi_x+b \phi_y) \\
    &= \text{ord}_{t=0} \ (a\phi_x(x(t),y(t))+b \phi_y(x(t),y(t))) \\
    &= \text{ord}_{t=0} \ \left( a\phi_x(x(t),y(t))-\frac{b  \phi_x(x(t),y(t))x'(t)}{y'(t)} \right) \\
    &= \text{ord}_{t=0} \ \left( \frac{a\phi_x(x(t),y(t))y'(t)-b  \phi_x(x(t),y(t))x'(t)}{y'(t)} \right) \\
    &= \text{ord}_{t=0} \ \left( \frac{\phi_x(x(t),y(t))}{y'(t)} (ay'(t)-bx'(t))\right) \\
    &= \text{ord}_{t=0} \ \varphi^* \phi_x +\text{ord}_{t=0} \ (ay'(t)-bx'(t))-\text{ord}_{t=0} \ y'(t) \\
    &= \text{ord}_{t=0} \ \varphi^* \phi_x +\min \{ \text{ord}_{t=0} \ y(t), \text{ord}_{t=0} \ x(t) \}-\text{ord}_{t=0} \ y(t),
\end{align*}
and
\begin{eqnarray*}
    i_0(\phi,l)&=& \text{ord}_{t=0} \ \varphi^* (ay-bx) \\
    &=& \text{ord}_{t=0} \ (ay(t)-bx(t)) \\
    &= &\min \{ \text{ord}_{t=0} \ y(t), \text{ord}_{t=0} \ x(t) \}.
\end{eqnarray*}
Finally, by rewriting (\ref{eqtang}) using that and (\ref{teissier}), we get
\begin{align*}
    \operatorname{tang}(\F,X,0) &=\text{ord}_{t=0} \ \varphi^*\omega+ \text{ord}_{t=0} \ \varphi^*\phi_x - \text{ord}_{t=0} \ y(t) +1 \\
    &=\text{ord}_{t=0} \ \varphi^*\omega +i_0(\phi,\mathcal{P}_l(\phi))-i_0(\phi,l)+1 \\
    &=\text{ord}_{t=0} \ \varphi^*\omega +\mu_0(X).
\end{align*}
\end{proof}
\begin{remark}
Let $X=\{\phi=0\}$ be a germ of an irreducible reduced curve and $f\in\mathcal{O}_2$ is a germ with an isolated singularity over $(X,0)$. Note that applying Proposition \ref{tangirred} to foliation $\F:\omega=df$, we obtain an expression to compute the Milnor number of the isolated complete intersection singularity defined by $(\phi,f)$ in the sense of Hamm \cite{Hamm}. Indeed, it follows from (\ref{Le}) that
\[\text{ord}_{t=0}\  \varphi^* \omega+\mu_0(X)=\operatorname{tang}(\F,X,0)=\operatorname{Ind}_{\operatorname{GSV} }(df;X,0)=\mu_0(X)+\mu_0(\phi,f).\]
Hence, $\mu_0(\phi,f)= \ord_{t=0}\  \varphi^* (df)$, where $\varphi$ is a Puiseux parametrization of $X$. 
\end{remark}
\subsection{The Bruce--Roberts number of a foliation and blow-ups}
When we deal with a simple blow-up $\pi$ at $0\in\C^2$, we say that $\pi$ is \textit{dicritical} relative to a holomorphic foliation $\F$ on $(\C^2,0)$, if the exceptional divisor $\pi^{-1}(0)$ is not $\tilde{\F}$-invariant. Otherwise, $\pi$ is called \textit{non-dicritical}.
\par  Let $X$ be a germ of a reduced curve at $(\C^2,0)$. We set \[\Omega_{X}:=\Omega^1_{\mathbb{C}^2,0}/(I_X\Omega^1_{\C^2,0}+\mathcal{O}_2 dI_X).\] the holomorphic 1-forms on $(X,0)$. 
\par In the following proposition, we show how the Bruce--Roberts number in dimension $2$ behaves under a blow-up. We denote by $\mu_{BR}(\F,X,p)$ the Bruce--Roberts number of $\F$ along $X$ around a neighborhood of the point $p$ (in Corollary \ref{dim2}, for example, it is defined around $0 \in \C^2$). 
\begin{proposition} \label{blowup}
Let $\F$ be a germ of a singular holomorphic foliation at $(\C^2,0)$, let $X$ be a germ of an irreducible reduced curve and let $\pi:\tilde{\C}^2 \rightarrow (\C^2,0)$ be the blow-up at $(\C^2,0)$. Assume that $X$ is not invariant by $\F$, $\tilde{\F}:=\pi^* \F$, and $q \in \pi^{-1}(0) \cap \tilde{X}$, where $\tilde{\F}$ and $\tilde{X}$ are the strict transforms of $\F$ and $X$ respectively. Denote by $m$ the multiplicity of $X$ in its Puiseux parametrization, and denote by $\nu$ the  algebraic multiplicity of the foliation $\F$ at $0\in\C^2$. Then, we have the following statements:
\begin{itemize}
\item[(a)] If $\pi$ is non-dicritical, then
\begin{align*}
\mu_{BR}(\F,X,0)&=\mu_{BR}(\tilde{\F},\tilde{X},q)+\nu^2-\nu-1+\nu m \\
&+\displaystyle \sum_{\substack{p \in \pi^{-1}(0) \\ p \neq q}}  \mu_p(\tilde{\F}) +\frac{m(m-1)}{2}-\mathcal{D};
\end{align*}
\item[(b)] If $\pi$ is dicritical, then
\begin{align*}
\mu_{BR}(\F,X,0)&=\mu_{BR}(\tilde{\F},\tilde{X},q)+\nu^2+\nu-1+(\nu+1)m \\
&+\displaystyle \sum_{\substack{p \in \pi^{-1}(0) \\ p \neq q}}  \mu_p(\tilde{\F}) +\frac{m(m-1)}{2}-\mathcal{D}.
\end{align*}
\end{itemize}
 We write $\mathcal{D}=\dim_{\C} \displaystyle\frac{\tilde{\sigma}^* \Omega_{\tilde{X}}}{\sigma^* \Omega_{X}}$, where $\tilde{\sigma}:(\overline{X},0) \rightarrow (\tilde{X},q)$ is the normalization, and $\sigma=\pi\circ\tilde{\sigma}$.
\end{proposition}

\begin{proof}
Let $\omega$  be the $1$-form defining $\F$ and let $\tilde{\omega}$ be the $1$-form defining $\tilde{\F}$. Applying Proposition \ref{tangirred} to the Corollary \ref{dim2}, we have

\begin{equation} \label{newbrob}
\mu_{BR}(\F,X,0)=\mu_0(\F)+\operatorname{ord}_{t=0} \varphi^* \omega +\mu_0(X)-\tau_0(X)
\end{equation}
and similarly,
\begin{align*}
\mu_{BR}(\tilde{\F},\tilde{X},q)=\mu_q(\tilde{\F})+\operatorname{ord}_{t=q} \tilde{\varphi}^* \tilde{\omega}+\mu_q(\tilde{X})-\tau_q(\tilde{X}),
\end{align*}
where $\varphi$ and $\tilde{\varphi}$ are the Puiseux parametrization of $X$ and $\tilde{X}$, respectively. 

The proof is obtained using results that show how the indices on the right side of equation (\ref{newbrob}) change under a blow-up. From \cite[Proposition 4.13]{CanoCerv}, we have
\begin{align*} 
\mu_0(\F)= \left\{ \begin{array}{cc}
\nu^2-(\nu+1)+\displaystyle\sum_{p \in \pi^{-1}(0)} \mu_p(\tilde{\F}) & \text{if} \ \pi \ \text{is non-dicritical;} \\
(\nu+1)^2-(\nu+2)+\displaystyle\sum_{p \in \pi^{-1}(0)} \mu_p(\tilde{\F}) & \text{if} \ \pi \ \text{is dicritical.}
\end{array} \right.
\end{align*}
and, from \cite[p. 4]{Wang}, we have
\begin{align*} 
\mu_0(X)-\mu_q(\tilde{X})=m(m-1) \ \text{and} \ \tau_0(X)-\tau_q(\tilde{X})=\displaystyle\frac{m(m-1)}{2}+\mathcal{D}.
\end{align*}
Now, from \cite[Section 2]{FPGBSM}, evaluating the $1$-form that defines $\tilde{\F}$ in $\tilde{\varphi}$ and taking orders, we get
\begin{align*}
\operatorname{ord}_{t=0} \varphi^* \omega= \left\{ \begin{array}{cc}
\nu m+ \operatorname{ord}_{t=q} \tilde{\varphi}^* \tilde{\omega} & \text{if} \ \pi \ \text{is non-dicritical;} \\
(\nu+1) m+\operatorname{ord}_{t=q} \tilde{\varphi}^* \tilde{\omega} & \text{if} \ \pi \ \text{is dicritical.}
\end{array} \right.
\end{align*}
Substituting the above formulas into (\ref{newbrob}), we conclude the proof of  the Proposition.
\end{proof}
The following examples illustrate  Proposition \ref{blowup}:
\begin{example}
Let $\F$ be a foliation defined by the $1$-form $\omega=2xdy-3ydx$, and let $X$ be the curve given by $X=\{ \phi=y^2-x^5=0 \}$. Observe that $X$ is not invariant by $\F$. If $\pi$ is a blow-up at $0\in\C^2$, in local coordinates, we have
\begin{align*}
\tilde{\F}_1 &=\pi^* \omega=x(-tdx+2xdt) \ \ \text{and} \\
\tilde{\F}_2 &=\pi^* \omega=y(-3ydu-udy),
\end{align*}
where $\tilde{\F}_1$ is obtained using the local chart $\pi(x,t)=(x,tx)$, and $\tilde{\F}_2$ is obtained using the local chart $\pi(u,y)=(uy,y)$. Now, we consider $\tilde{X}=\{ t^2-x^3=0 \}$, since
\begin{align*}
\phi \circ  \pi(x,t)=\phi(x,tx)=x^2(t^2-x^3).
\end{align*}
Since $\tilde{X}$ is not invariant by $\tilde{\F}_1$, set $\tilde{\F}=\tilde{\F}_1$. Thus, $\tilde{\F}$ is defined by the $1$-form $\tilde{\omega}$, given by \[\tilde{\omega}=-tdx+2xdt.\] 
Using Singular (\cite{DGPS}), we get $\mu_{BR} (\F,X,0)=7$, $\mu_{BR} (\tilde{\F},\tilde{X},q)=5$, $\nu=1$, $m=2$, $\displaystyle\sum_{\substack{p \in \pi^{-1}(0) \\ p \neq q}}  \mu_p(\tilde{\F})=\mu_0(\tilde{\F}_2)=1$ and $\mathcal{D}=1$. Note that
\begin{align*}
\mu_{BR}(\F,X,0) &=\mu_{BR}(\tilde{\F},\tilde{X},q)+\nu^2-\nu-1+\nu m + \mu_0(\tilde{\F}_2) +\frac{m(m-1)}{2}-\mathcal{D} \\
&=5+1-1-1+2 \cdot 1+1+1-1=7.
\end{align*}
Since $\pi$ is non-dicritical, item $(a)$ of Proposition \ref{blowup} is satisfied.
\end{example}

\begin{example}
Consider the foliation $\F$ defined by
\[\omega=(2x^7+5y^5)dx-xy^2(5y^2+3x^5)dy\]
and the curve $X=\{\phi=y^3-x^7=0\}$. In the charts $\pi_1(x, t) = (x, xt)$ and $\pi_2(u, y) = (uy, y)$, direct
substitution gives
\[\pi_1^{*}\omega=x^6[x(2-3xt^3)dx-t^2(5t^2+3x^3)dt]\]
\[\pi_2^{*}\omega=y^6[(5+2u^7y^2)du+u^6y(2u^2-3y)dy]\]
After removing the exceptional factors, these forms define the strict transform $\tilde{\F}$. The exceptional
divisor is not invariant, and its only singular point is $q = (0, 0)$ in the first chart. The strict
transform of $X$ is $\tilde{X} = \{t^3-x^4 = 0\}$. One obtains
\[\mu_0(\F) = 33,\qquad \mu_q(\tilde{\F}) = 4,\qquad \mu_{BR}(\F, X,0) = 56,\qquad \mu_{BR}(\tilde{\F}, \tilde{X}, q ) = 9.\]
Moreover, $\nu= 5$, $m = 3$, $\mathcal{D} = 3$ and there are no singularities of $\tilde{\F}$ on the exceptional divisor
outside $q$. Therefore, Proposition \ref{blowup}(b) gives
\[\mu_{BR}(\F,X,0)=56=9+25+5-1+6\cdot 3+0+3-3,\]
as required.
\end{example}

\subsection{The relative Bruce--Roberts number for foliations on $(\C^2,0)$}
For holomorphic foliations on $(\C^2,0)$, Theorem \ref{correlthmBr} can be stated as follows:
\begin{corollary}\label{red2}
Let $\F$ be a germ of a singular holomorphic foliation at $0\in\mathbb{C}^2$. Let $X$ be a germ of a reduced curve at  $0\in\C^2$.
Assume that none of the irreducible components of $X$ is invariant by $\F$. Then
    \[\mu_{BR}(\F,X)=\mu_0(\F)+\mu_{BR}^-(\F,X).\]
\end{corollary}
The following examples illustrate Corollary \ref{red2}:
\begin{example}
    We can use Singular \cite{DGPS} to compute the relative Bruce--Roberts number for the previous examples:
    \begin{itemize}
        \item In Example \ref{ex1}, we have $\F:\omega=\lambda xdy+ydx=0, \ \lambda \neq -p/q$, and $X=\{ \phi=y^p-x^q=0 \}$. As we needed numerical examples, for $\lambda=1$, we have
        \begin{align*}
            &\mu_{BR}^-(\F,X)=6, \ \text{when} \ p=2 \ \text{and} \ q=5, \ \text{and} \\ &\mu_{BR}^-(\F,X)=23, \ \text{when} \ p=11 \ \text{and} \ q=13.
        \end{align*}
        In both cases, $\mu_{BR}^-(\F,X)=p+q-1=\mu_{BR}(\F,X)-\mu_0(\F)$. \\
         \item In Example \ref{ex3}, we have $\F_{\omega}:\omega=(y^3+y^2-xy)dx-(2xy^2+xy-x^2)dy=0$, $\F_{\eta}:\eta=(2y^2+x^3)dx-2xydy=0$ and $X=\{ \phi=y^7-x^3=0 \}$, with $\mu_{BR}(\F_{\omega},X)=\mu_{BR}(\F_{\eta},X)=17$ and $\mu_0(\F_{\omega})=\mu_0(\F_{\eta})=5$. Then
        \begin{align*}
            \mu_{BR}^-(\F_{\omega},X)=\mu_{BR}^-(\F_{\eta},X)=12.
        \end{align*}
  
    \end{itemize}
    In each case, we have $\mu_{BR}^-(\F,X)=\mu_{BR}(\F,X)-\mu_0(\F)$, satisfying Corollary \ref{red2}.
\end{example}
\par An interesting property for the Milnor number of a foliation $\F$ defined by $\omega=Adx+Bdy$ is 
\begin{equation}\label{nu1}
\mu_0(\F)=i_0(A,B)\geq \ord(A)\ord(B)\geq \nu^2,
\end{equation}
where $\nu$ denotes the algebraic multiplicity of $\F$.
\par As an application of the above inequality, we have
\begin{corollary}
Let $\F$ be a germ of a singular holomorphic foliation at $0\in\mathbb{C}^2$. Let $X$ be a germ of a reduced curve at  $0\in\C^2$.
 Assume that none of the irreducible components of $X$ is invariant by $\F$. Then
 \[\mu_{BR}(\F,X)-\mu_{BR}^-(\F,X)\geq \nu^2,\]
 where $\nu$ denotes the algebraic multiplicity of $\F$.
\end{corollary}
To end this subsection, we present a blow-up formula for the relative Bruce--Roberts number of a foliation $\F$ relative to a non-invariant curve $X$.
\begin{corollary}\label{bruce}
Let $\F$ be a germ of a singular holomorphic foliation at $(\C^2,0)$, let $X$ be a germ of an irreducible reduced curve and let $\pi:\tilde{\C}^2 \rightarrow (\C^2,0)$ be the blow-up at $(\C^2,0)$. Assume that $X$ is not invariant by $\F$, $\tilde{\F}:=\pi^* \F$, and $q \in \pi^{-1}(0) \cap \tilde{X}$, where $\tilde{\F}$ and $\tilde{X}$ are the strict transforms of $\F$ and $X$ respectively. Denote by $m$ the multiplicity of $X$ in its Puiseux parametrization, and denote by $\nu$ the  algebraic multiplicity of the foliation $\F$ at $0\in\C^2$. Then, we have the following statements:
\begin{itemize}
\item[(a)] If $\pi$ is non-dicritical, then
\begin{align*}
\mu_{BR}^-(\F,X,0)&=\mu_{BR}^-(\tilde{\F},\tilde{X},q)+\nu m +\frac{m(m-1)}{2}-\mathcal{D};
\end{align*}
\item[(b)] If $\pi$ is dicritical, then
\begin{align*}
\mu_{BR}^-(\F,X,0)&=\mu_{BR}^-(\tilde{\F},\tilde{X},q)+(\nu+1)m +\frac{m(m-1)}{2}-\mathcal{D}.
\end{align*}
\end{itemize}
We write $\mu_{BR}^-(\F,X,p)$ to denote the relative Bruce--Roberts number around the point $p$. Moreover, we have $\mathcal{D}=\dim \displaystyle\frac{\tilde{\sigma}^* \Omega_{\tilde{X}}}{\sigma^* \Omega_{X}}$, where $\tilde{\sigma}:(\overline{X},0) \rightarrow (\tilde{X},q)$ is the normalization, and $\sigma=\pi\circ\tilde{\sigma}$.
\end{corollary}

\begin{proof}
    The proof follows similarly, utilizing the results from the proof of Proposition \ref{blowup} together with Corollary \ref{red2}.
\end{proof}

\subsection{Generalized curve foliations and the Bruce--Roberts numbers}
 We say that $0\in\C^2$ is a \textit{reduced} singularity for $\F$ if the linear part $\text{D}v(0)$ of the vector field $v$ in (\ref{oneform}) is non-zero and has eigenvalues $\lambda_1,\lambda_2\in\C$ fitting in one of the cases:
\begin{enumerate}
\item[(i)] $\lambda_1\lambda_2\neq 0$ and $\frac{\lambda_1}{\lambda_2}\not\in\mathbb{Q}^{+}$ (\textit{non-degenerate});
\item[(ii)] $\lambda_1\neq 0$ and $\lambda_2= 0$ (\textit{saddle-node singularity}).
\end{enumerate}
In the case $(i)$, there is a system of coordinates  $(x,y)$ in which $\F$ is defined by the equation
\begin{equation}
\label{non-degenerate}
\omega=x(\lambda_1+a(x,y))dy-y(\lambda_2+b(x,y))dx,
\end{equation}
where $a(x,y),b(x,y)  \in {\mathbb C}[[x,y]]$ are non-units, so that  $Sep_0(\F)$ is formed by two
transversal analytic branches given by $\{x=0\}$ and $\{y=0\}$. In the case $(ii)$, up to a formal change of coordinates, the  saddle-node singularity is given by a 1-form of the type
\begin{equation}
\label{saddle-node-formal}
\omega = x^{k+1} dy-y(1 + \lambda x^{k})dx,
\end{equation}
where $\lambda \in \mathbb{C}$ and $k \in \mathbb{Z}^{+}$ are invariants after formal changes of coordinates (see \cite[Proposition 4.3]{martinetramis}).
The curve $\{x=0\}$   is an analytic separatrix, called {\em strong} separatrix, whereas $\{y=0\}$  corresponds to a possibly formal separatrix, called {\em weak} separatrix. The integer $k+1>1$ is called \textit{tangency index} of $\F$ relative to the weak separatrix.
\par According to \cite[p. 144]{CLS}, a \textit{generalized curve foliation} is a foliation $\F$ defined by a vector field $v$ whose reduction of singularities (\cite{seidenberg}) admits only non-degenerate singularities with non-vanishing eigenvalues. Assuming that $\F$ is non-dicritical at $0\in\C^2$, and denoting $C=Sep_0(\F)$ the union of the separatrices of $\F$, Camacho-Lins Neto-Sad proved the following theorem:
\begin{theorem}{\cite[Theorem 4]{CLS}}\label{camacho}
Given a germ of a non-dicritical holomorphic foliation $\F$ at $0\in\C^2$ one has $\mu_0(\F)\geq\mu_0(C)$, and the equality holds if and only if $\F$ is a generalized curve foliation.  
\end{theorem}
As a consequence of these results, we obtain a new characterization of non-dicritical generalized curve foliations. Recall that, in this case, all separatrices of \(\mathcal{F}\) are analytic.
\begin{corollary}\label{generalized}
Let $\F$ be a germ of a non-dicritical holomorphic foliation at $0\in\mathbb{C}^2$. Let $X$ be a germ of a reduced curve at  $0\in\C^2$, and $C=Sep_0(\F)=\{f=0\}$ be a reduced equation of $Sep_0(\F)$. 
 Assume that none of the irreducible components of $X$ is invariant by $\F$. Then
 \[\mu_{BR}(\F,X)-\mu_{BR}(f,X)\geq \mu^{-}_{BR}(\F,X)-\mu^{-}_{BR}(f,X),\]
and the equality holds if and only if $\F$ is a generalized curve foliation.
\end{corollary}
\begin{proof}
The conclusion follows from
 Corollary \ref{red2} and Theorem \ref{camacho}.
\end{proof}
\section{Applications to global foliations}\label{global}
 \par Let $S$ be a compact complex surface and let $\{U_j\}_{j\in I}$ be an open covering of $S$. A \textit{holomorphic foliation} $\mc{F}$ on $S$ can be described by a collection of holomorphic 1-forms $\omega_j\in\Omega^{1}_{S}(U_{j})$ with isolated zeros such that 
\begin{equation*}
\omega_i=g_{ij}\omega_j\,\,\,\,\,\,\,\,\,\text{on}\,\,\,U_i\cap U_j,\,\,\,\,\,\,\,\,\,\,\,\,\,\,g_{ij}\in\mathcal{O}^{*}_{S}(U_i\cap U_j).
\end{equation*}
The \textit{singular set} $\sing(\mc{F})$ of $\mc{F}$ is the finite subset of $S$ defined by 
$$\sing(\mc{F})\cap U_{j}=\text{zeros of}\,\,\,\omega_{j},\,\,\,\,\,\,\,\,\,\,\,\,\,\,\,\forall j\in I.$$
A point $q\not\in\sing(\mc{F})$ is said to be \textit{regular}. The cocycle $\{g_{ij}\}$ defines a line bundle $N_{\mc{F}}$ on $S$, called \textit{the normal bundle} of $\mc{F}$.  The dual $N^{*}_{\F}$ is called \textit{conormal bundle} of $\F$. The foliation gives rise to a global holomorphic section of $N_\F\otimes T^{*}S$, with finite zero set and modulo multiplication by $\mathcal{O}^{*}_S(S)$, and to an exact sequence
\[0\to N^{*}_{\F}\to T^{*}S\to \mathcal{I}_Z\cdot T^{*}_{\F}\to 0\]
where $T^{*}_{\F}$ is a line bundle on $S$, called \textit{canonical bundle} of $\F$, and $\mathcal{I}_Z$ is an ideal sheaf supported on $\sing(\F)$. The dual $T_{\F}$ of $T^{*}_\F$, is called \textit{the tangent bundle} of $\F$.
These line bundles are related to each other via the canonical bundle $K_S$ of $S$:
\[K_S=T^{*}_\F\otimes N^{*}_\F.\]
Now we can state the following theorem.

\begin{theorem}\label{global_1}

Let $\F$ be a holomorphic foliation on a compact complex surface $S$, and let $X\subset S$ be a reduced compact curve, none of whose irreducible components is invariant by $\F$. Then
\[ \sum_{p\in X}\mu_{BR}(\F,X,p)=\sum_{p\in\sing(\F)\cap X}\mu_p(\F)+N_\F\cdot X-\chi(X)-\tau(X). \]
In particular,
\[ \begin{aligned} \sum_{p\in\sing(\F)\cap X}\mu_{BR}(\F,X,p) \leq{}& T_\F\cdot T_\F+T_\F\cdot K_S+c_2(S)\\ &+N_\F\cdot X-\chi(X)-\tau(X)\\ &-\sum_{p\in\sing(\F)\setminus X}\mu_p(\F). \end{aligned} \]
Moreover,
\[ \tau(X)\leq N_\F\cdot X-\chi(X). \]
Here $\chi(X)=-K_S\cdot X-X\cdot X$ is the virtual Euler characteristic of $X$, and $\tau(X)=\sum_{p\in X}\tau_p(X)$ is its global Tjurina number. The local Bruce--Roberts numbers are defined by the same local-algebra formula at every point of $X$, including regular points of $\F$, where $\mu_p(\F)=0$. All sums over $X$ have only finitely many nonzero terms.

\end{theorem}

\begin{proof}
For every $p\in X$, we have
\begin{equation}\label{eq_final} \mu_{BR}(\F,X,p)=\mu_p(\F)+\operatorname{tang}(\F,X,p)-\tau_p(X). \end{equation}
At singular points of $\F$, this follows from Corollary~\ref{dim2}. Note that for a regular point of $\F$ we have $\mu_{BR}(\F,X)=\operatorname{tang}(\F,X,p)-\tau_p(X)$.
According to~\cite[Propositions~2.1 and~2.2]{Brunella-book},

\[ \sum_{p\in\sing(\F)}\mu_p(\F)=T_\F\cdot T_\F+T_\F\cdot K_S+c_2(S) \]

and

\[ \sum_{p\in X}\operatorname{tang}(\F,X,p)=N_\F\cdot X-\chi(X). \]
Summing~\eqref{eq_final} over $X$ gives the asserted identity. Since the local Bruce--Roberts numbers are nonnegative,
\[ \sum_{p\in\sing(\F)\cap X}\mu_{BR}(\F,X,p)\leq\sum_{p\in X}\mu_{BR}(\F,X,p). \]
Combining this with
\[ \sum_{p\in\sing(\F)\cap X}\mu_p(\F)=T_\F\cdot T_\F+T_\F\cdot K_S+c_2(S)-\sum_{p\in\sing(\F)\setminus X}\mu_p(\F) \]
proves the stated inequality.
Finally, if $\omega=A\,dx+B\,dy$ is a local defining form of $\F$, then $\omega(\Theta_X)\subset\langle A,B\rangle$. Therefore $\mu_{BR}(\F,X,p)\geq\mu_p(\F)$ at every $p\in X$. Summing these inequalities and applying the global identity gives
\[ 0\leq\sum_{p\in X}\bigl(\mu_{BR}(\F,X,p)-\mu_p(\F)\bigr)=N_\F\cdot X-\chi(X)-\tau(X), \]
which proves the final assertion.
\end{proof}
 

\par On the other hand, since \[\sum_{p\in\sing(\F)}\mu_p(\F)=T_\F \cdot T_{\F}+T_{\F}\cdot K_S+c_2(S)\] by \cite[Proposition 2.1]{Brunella-book}, we have the following corollary as a consequence of Corollary \ref{red2}.
\begin{corollary}
Let $\F$ be a holomorphic foliation on a compact complex surface $S$, and let $X\subset S$ be a reduced compact curve, none of whose components are invariant by $\F$. Then
\begin{eqnarray*}
    \sum_{p\in\sing(\F)\cap X}\left[\mu_{BR}(\F,X,p)-\mu_{BR}^{-}(\F,X,p)\right]&=&\sum_{p\in\sing(\F)\cap X}\mu_p(\F)\\
    &=& T_\F\cdot T_\F+T_\F\cdot K_S+c_2(S)-\\ & &\sum_{p\in\sing(\F)\setminus X}\mu_p(\F).
\end{eqnarray*}
\end{corollary}

\vspace{2cm}

\noindent{\bf Data Availability Statement:} 
Data sharing is not applicable to this article as no data sets were generated or analyzed during the current study.
\vspace{1cm}

\noindent{\bf Declarations
Conflict of Interest:} The authors declare that they have no conflict of interest.

\end{document}